\documentclass[a4paper,11pt,reqno]{amsart}
\usepackage[utf8]{inputenc}
\usepackage[T1]{fontenc}
\usepackage[top=1.3in,bottom=1.3in,left=1.25in,right=1.25in]{geometry}
\usepackage{setspace}
\usepackage{pgfplots}
\usepackage{amssymb,amsmath,amsthm,graphicx,xcolor,mathtools,mathrsfs}
\usepackage{comment}

\usepackage{hyperref}
\hypersetup{
	colorlinks,
	linkcolor={red!50!black},
	citecolor={green!50!black},
	urlcolor={blue!50!black}
}

\usepackage[
backend=biber,
natbib=true,
style=numeric,
isbn=false,
maxbibnames=9
]{biblatex}
\usepackage{csquotes}
\renewbibmacro{in:}{}
\DeclareFieldFormat{pages}{#1}

\DeclareFieldFormat{title}{\mkbibquote{#1}}
\DeclareFieldFormat[misc]{date}{Preprint (#1)}

\renewcommand{\nrightarrow}{\not\rightarrow}

\newcommand{\oC}{\overline{C}}

\newcommand{\pt}{q}
\newcommand{\pp}{p}
\newcommand{\ho}{\mathcal{H}_0}
\newcommand{\hnp}{\mathcal{H}_{s}(n,p)}
\newcommand{\notRamsey}{\nrightarrow (K_{s},K_{k-1})}
\newcommand{\cA}{\mathcal{A}}
\newcommand{\cH}{\mathcal{H}}
\newcommand{\cE}{\mathcal{E}}

\newcommand{\cC}{\mathcal{C}}

\newcommand{\PP}[1]{\mathbb{P} \left[ #1 \right]}
\newcommand{\EE}[1]{\mathbb{E} \left[ #1 \right]}

\renewbibmacro*{doi+eprint+url}{
	\iftoggle{bbx:url}
	{\iffieldundef{doi}{\usebibmacro{url+urldate}}{}}{}
	\newunit\newblock
	\iftoggle{bbx:eprint}
	{\usebibmacro{eprint}}
	{}
	\newunit\newblock
	\iftoggle{bbx:doi}
	{\printfield{doi}}
	{}
}

\usepackage{cleveref}
\theoremstyle{plain}
\newtheorem{theorem}{Theorem}[section]
\crefname{theorem}{Theorem}{Theorems}

\crefname{proposition}{Proposition}{Propositions}

\crefname{corollary}{Corollary}{Corollaries}

\newtheorem{lemma}[theorem]{Lemma}
\crefname{lemma}{Lemma}{Lemmas}

\newtheorem{conjecture}[theorem]{Conjecture}
\crefname{conjecture}{Conjecture}{Conjectures}

\crefname{problem}{Problem}{Problem}

\newtheorem{claim}[theorem]{Claim}
\crefname{claim}{Claim}{Claims}

\crefname{observation}{Observation}{Observations}

\crefname{setup}{Setup}{Setups}

\crefname{myth}{Myth}{Myths}

\newtheorem{fact}[theorem]{Fact}
\crefname{fact}{Fact}{Facts}

\crefname{algorithm}{Algorithm}{Algorithms}

\crefname{remark}{Remark}{Remarks}

\crefname{example}{Example}{Examples}

\theoremstyle{definition}

\crefname{definition}{Definition}{Definitions}

\crefname{construction}{Construction}{Constructions}

\crefname{question}{Question}{Questions}

\numberwithin{equation}{section}

\usepackage[shortlabels]{enumitem}
\setlist[enumerate,1]{label={\upshape (\roman*)}}

\let\epsilon=\varepsilon
\allowdisplaybreaks

\newenvironment{proofclaim}[1][Proof of the claim]{\begin{proof}[#1]}{\end{proof}}

\title{Graphs with asymmetric Ramsey properties}
\author{Walner Mendonça}
\email{walner@mat.ufc.br}
\address[UFC]{Departamento de Matemática, Universidade Federal
  do Ceará, Ceará, Brazil}
\author{Meysam Miralaei}
\email{m.miralaei@ime.usp.br}
\author{Guilherme O. Mota}
\email{mota@ime.usp.br}
\address[USP]{Instituto de Matem\'atica e Estat\'istica,
  Universidade de S\~ao Paulo, Rua do Mat\~ao 1010, 05508--090
  São Paulo, Brazil}

\thanks{This research was partly supported by CAPES (Finance Code 001).
    W.\ Mendonça was supported by CNPq (312935/2025-0), FAPESP (2023/07695-6),  and FAPESB (012/2022 - UNIVERSAL - APP0044/2023).
    M.\ Miralaei was supported by FAPESP (2023/04895-4).
    G.\ O.\ Mota was supported by CNPq (315916/2023-0 and 406248/2021-4) and FAPESP (2023/03167-5 and 2024/13859-4).
    }

\begin{document}
\onehalfspace

\begin{abstract}
  Given positive integers $k$ and $\ell$ we write
  $G \rightarrow (K_k,K_\ell)$ if every 2-colouring of the edges of
  $G$ yields a red copy of $K_k$ or a blue copy of $K_\ell$ and we
  denote by $R(k)$ the minimum $n$ such that
  $K_n\rightarrow (K_k,K_k)$. By using probabilistic methods and
  hypergraph containers we prove that for every integer $k \geq 3$,
  there exists a graph $G$ such that $G \nrightarrow (K_k,K_k)$ and
  $G \rightarrow (K_{R(k)-1},K_{k-1})$.  This result can be viewed as
  a variation of a classical theorem of Ne\v{s}et\v{r}il and R\"odl
  [The Ramsey property for graphs with forbidden complete subgraphs,
  {\em Journal of Combinatorial Theory, Series B}, \textbf{20} (1976),
  243–249], who proved that for every integer $k\geq 2$ there exists a
  graph $G$ with no copies of $K_k$ such that
  $G\rightarrow(K_{k-1}, K_{k-1})$.
\end{abstract}

\maketitle

\section{Introduction}
Given positive integers $k$ and $\ell$, we say a graph $G$ is
\emph{Ramsey} for $(K_k,K_\ell)$ if every colouring of the edges of
$G$ with red and blue contains a red copy of $K_k$ or a blue copy of
$K_\ell$ and we denote this property by $G \rightarrow (K_k,K_\ell)$.
In a seminal work~\cite{Ramsey}, Ramsey proved that for all positive
integers $k$ and $\ell$, there exists a positive integer $n$ such that
$K_n \rightarrow (K_k, K_{\ell})$.  In the special case $k=\ell$, we
simply write $G \rightarrow K_k$ and we define the \emph{Ramsey
  number} $R(k)$ as the minimum $n$ such that $K_n \rightarrow K_k$.

Estimating $R(k)$ has proven to be notoriously difficult and a central
problem in Ramsey theory is to determine the Ramsey number $R(k)$.
Classical results due to Erd\H{o}s~\cite{Erd} and
Erd\H{o}s–-Szekeres~\cite{Erd.Szek} established the bounds
$2^{k/2} \leq R(k) \leq 2^{2k}$.  Despite several refinements, these
exponents remained essentially unchanged for decades. Only in recent
years, significant breakthroughs have been achieved (see,
e.g.,~\cite{Conlon, ConlonFoxSudakov2015, Sah2023}). In a striking
advance, Campos, Griffiths, Morris, and Sahasrabudhe~\cite{CGMS}
proved that there exists an $\varepsilon > 0$ such that
$R(k)\leqslant (4-\varepsilon)^k$ for sufficiently large $k$.  This
result provides an exponential improvement over the classical
Erd\H{o}s–-Szekeres upper bound. More generally, one can think of the
minimum $n$ such that $K_n\rightarrow (K_k,K_\ell)$, for which there
was another major breakthrough recently by Mattheus and
Verstraete~\cite{MV}, who showed that
$n=\Omega\big(t^3/(\log^4 t)\big)$ vertices are enough to force red
copies of $K_4$ or blue copies of $K_t$ in red-blue colourings of the
edges of $K_n$.

Although much effort has been put into estimating Ramsey numbers, a
parallel and rich direction of research investigates the structure of
graphs that are Ramsey for given pairs of graphs. In this context, we
study Ramsey phenomena of the form $G \rightarrow (K_{s},K_{t})$.

A classical result of  Ne\v{s}et\v{r}il and R\"odl~\cite{NR} shows
that for every $k\geq 2$ there are graphs with no copies of $K_k$
that are Ramsey for $K_{k-1}$.
\begin{theorem}[Ne{\v{s}}et{\v{r}}il \& R{\"o}dl, 1976]\label{thm_NR}
	For every $k\geq 2$ there is a graph $G$ such that
	$K_k \nsubseteq G$ and $G\rightarrow K_{k-1}$.
\end{theorem}
Our main result, Theorem~\ref{thm:main} below, can be seen as a
variation of Theorem~\ref{thm_NR}. We prove that for any $k\geq 3$
there exists a graph $G$ that is not Ramsey for $K_k$ but it is Ramsey
for the pair $(K_{s},K_{k-1})$, for $s=R(k)-1$, i.e., we replace the
condition $K_k \nsubseteq G$ in Theorem~\ref{thm_NR} with the weaker
condition $G\nrightarrow K_k$, which allows $G$ to contain copies of
$K_k$, but still there is a colouring of $E(G)$ avoiding monochromatic
copies of $K_k$; and we strengthen the conclusion
$G\rightarrow K_{k-1}$ by showing that
$G \rightarrow (K_{s},K_{k-1})$, for $s=R(k)-1$ (note that $s$ cannot
be any larger).
\begin{theorem}\label{thm:main}
	For every integer $k \ge 3$, there exists a graph $G$ such that $G
	\nrightarrow K_{k}$ and $G \rightarrow (K_{s},K_{k-1})$, for
  $s=R(k)-1$.
\end{theorem}

We remark that Theorem~\ref{thm:main} also relates to the theory of
Ramsey equivalence. Szab\'o, Zumstein, and Z\"urcher \cite{Szabo}
introduced the notion of \emph{Ramsey-equivalent graphs}: two graphs
$H_1$ and $H_2$ are {Ramsey-equivalent} if for every graph $G$, we
have $G \rightarrow H_1$ if and only if $G \rightarrow H_2$
(see~\cite{Bloom, Fox}) for results on Ramsey equivalence). More
generally, two pairs of graphs $(F_1, H_1)$ and $(F_2, H_2)$ are
{Ramsey-equivalent} if for every graph $G$ we have
$G \rightarrow (F_1, H_1)$ if and only if $G \rightarrow (F_2, H_2)$.
In other words, the two pairs share exactly the same family of Ramsey
graphs. In this direction, our result implies that the pairs $(K_k,K_k)$
and $(K_{s},K_{k-1})$ for any $s\leq R(k)-1$ are not
Ramsey-equivalent.

The proof of Theorem~\ref{thm:main} combines probabilistic methods
with the hypergraph container
framework~\cite{balogh2015independent,SaxtonThomason2015} and is
inspired by ideas from~\cite{bollobas2001ramsey}. The rest of the
paper is organized as follows. In Section~\ref{sec:mono}, we show that
with high probability\footnote{Meaning with probability going to $1$
  as $n$ tends to infinity.} the graph $G$ obtained in a natural way from
every ``dense'' subhypergraph of a suitable $n$-vertex random
$s$-uniform hypergraph satisfies $G\rightarrow(K_{s},K_{k-1})$ for
$s = R(k)-1$. In Section~\ref{sec:hyper}, we show that with high
probability a suitable random hypergraph $\cH$ contains a dense
subhypergraph $\cH_0$ that will allow us to obtain a graph $G$ such
that $G\nrightarrow K_k$. These results are then combined in
Section~\ref{sec:proof}. Finally, in Section~\ref{sec:conc}, we
outline some directions for future research.

\section{Graphs induced by random hypergraphs}
\label{sec:mono}
In the remainder of the paper, we fix a positive integer $k \geq 3$
and put $s=R(k)-1$. In this section, we prove that suitable random
$s$-uniform hypergraphs induce a graph with Ramsey properties with
respect to $(K_s,K_{k-1})$, but before presenting this result we
briefly discuss the hypergraph container lemma and state some simple
facts that will be useful when analysing our construction.

\subsection{Hypergraph Containers and tools}

An important parameter in our analysis, which is also common in many
results in Ramsey Theory when describing Ramsey properties in random
graphs, is the \emph{maximum $2$-density} of a graph~$F$, defined as
\begin{equation*}
  m_2(F)=\max\left\{\frac{e(J)-1}{v(J)-2} \;:\; J\subset
    F,\;v(J)\geq 3\right\},
\end{equation*}
where $e(J)$ and $v(J)$ denote the number of edges and vertices of $J$, respectively.
We use the hypergraph container
lemma~\cite{balogh2015independent,SaxtonThomason2015} stated as in~\cite{NS16}
to obtain a set of \emph{containers} $C_i$ and their corresponding
\emph{sources} $S_i$.

\begin{lemma}[Container Lemma]
  \label{containers}
  For every graph $F$ and every $\delta > 0$, there exist $n_0$ and $D > 0$
  such that for all $n \geq n_0$ there exists $t=t(n)$ such that the
  following holds: there are pairwise distinct subsets
  $S_1,\dots,S_t \subseteq E(K_n)$ and
  $C_1,\dots,C_t \subseteq E(K_n)$ such that
  \begin{enumerate}
    \item \label{eq:cont1}$|S_i| \leq D n^{2 - 1/m_2(F)}$ for all $i$;
    \item \label{eq:cont2}each $C_i$ contains at most $\delta n^{v(F)}$ copies of $F$;
    \item \label{eq:cont3}for every $F$-free graph $G$ with \(n\) vertices, there exists $i$ such that $S_i \subseteq E(G) \subseteq C_i$.
  \end{enumerate}
\end{lemma}

In the proof of Theorem~\ref{thm:main} we apply Lemma \ref{containers}
together with the following simple supersaturation result (see,
e.g.,~\cite{NS16}) that guarantees many red copies of $K_{s}$ or many
blue copies of $K_{k-1}$ when colouring the edges of a sufficiently
large complete graph (with at least $R(s, k-1)$ vertices).

\begin{fact}
  \label{thm_RS}
  For all integers $s > k \geq 2$ there exists $\delta >0$ such that
  the following holds for sufficiently large $n$. Every red-blue
  colouring of the edges of $K_n$ contains more than $\delta n^{s}$
  red copies of $K_{s}$ or more than $\delta n^{k-1}$ blue copies of
  $K_{k-1}$.
\end{fact}

Let $\cH$ be a hypergraph and let $J$ be a graph with
$V(J)\subseteq V(\cH)$ and $E(J) = \{e_1,\ldots,e_m\}$.
We write \(J \sqsubset \cH\) if $\cH$ contains
distinct hyperedges \(E_1,\ldots,E_m \in E(\cH)\) such that
\(e_i \subseteq E_i\) for every \(i\in [m]\).  Given positive integers
$n$ and $k$ and a probability function $\pp=\pp(n)$, the \emph{random
  $s$-uniform hypergraph} $\hnp$ is the $n$-vertex $s$-uniform
hypergraph obtained by adding any possible hyperedge with $s$ vertices
independently with probability $\pp$. The following fact follows from
Markov's inequality.

\begin{fact}\label{lem:quasi}
  Let $k\geq 3$ be an integer and let $\pp\in(0,1)$ and $\cH=\hnp$.
Then, for every graph \(J\) with \(V(J) \subseteq V(\cH)\), we have
$\mathbb{P}[J \sqsubset \cH] \leq \pt^{e(J)}$, where $\pt = \pp \binom{n-2}{s-2}$.
\end{fact}
\begin{proof}
  Let \(E(J) = \{e_1,\ldots,e_m\}\) and let \(X\) be the number of
  \(m\)-tuples \((E_1,\ldots,E_m)\) consisting of $m$ distinct
  hyperedges of \(\cH\) such that \(e_i \subseteq E_i\) for each
  \(i \in [m]\).  Therefore, by Markov's inequality, we have 
   $\PP{J \sqsubset \cH} = \PP{X \geq 1} \leq \EE{X} \leq
   \binom{n-2}{s-2}^{m} \pp^m = \pt^{e(J)}$.
 \end{proof}

\subsection{Graphs Ramsey for ($K_{R(k)-1}$, $K_{k-1}$)}

Recall that $s = R(k)-1$ and consider an $s$-uniform hypergraph
$\cH$. We define the \emph{primal graph} $G[\cH]$ of $\cH$ as the
graph on the same vertex set as $\cH$ and edge set $E(G[\cH])$
consisting of all pairs of vertices that appear together in the same
hyperedge of $\cH$. The following theorem is the main result of this
section.

\begin{theorem}
  \label{thm:part1}
  For all integers $s\geq 2$ and $k\geq 3$, there exists \(C>0\) such that the
  following holds with high probability for $\cH = \hnp$ when
  $\pp \geq C n^{2-s-1/m_2(K_{k-1})}$. For every subhypergraph
  $\ho \subseteq \cH$ with at least $(1 - o(1))e(\cH)$ hyperedges, we
  have
  \[
  G[\ho] \rightarrow (K_{s}, K_{k-1}).
  \]
\end{theorem}

\begin{proof}
  Let $\delta = \delta(s,k) > 0$ be given by Fact~\ref{thm_RS} and
  apply Lemma~\ref{containers} with $F = K_{k-1}$ and $\delta$ to
  obtain $D > 0$, an integer \(n_0\), a collection
  \(S_1,\ldots,S_t \subseteq E(K_n)\) of distinct sources and a
  collection \(C_1,\ldots,C_t\subseteq E(K_n)\) of containers.  Let
  $n$ be sufficiently large and consider
  $\oC_i = E(K_n) \setminus C_i$ for every $i\in [t]$.

  From Lemma~\ref{containers}\ref{eq:cont2}, each $C_i$ contains at
  most $\delta n^{k-1}$ copies of $K_{k-1}$ in \(K_n\), which by
  Fact~\ref{thm_RS} implies that the edges of more than $\delta n^{s}$
  copies of \(K_{s}\) are in $\oC_i$. For each $i\in[t]$, let $\cA_i$
  be the collection of the vertex set of those copies of $K_{s}$.

  Finally, let $C=C(D,\delta)$ be sufficiently large and
  $\pp \geq C n^{2-s-1/m_2(K_{k-1})}$ and let \(\cH = \hnp\).  We will
  show that the probability that there exists \(\ho \subseteq \cH\)
  with \(e(\ho) \geq (1-\delta)e(\cH)\) such that \(G[\ho]\notRamsey\)
  is sufficiently small for our purposes. In the following claim we
  reduce this event to another event which is entirely described in
  terms of the sources and the containers. For each
  \(i \in [t]\), let $X_i$ be the number of sets in the collection
  \(\cA_i\) which are hyperedges in \(\cH\), that is
  \[
    X_i = |\{ A \in \cA_i : A \in E(\cH)\}|.
  \]
  \begin{claim}
    If there exists a subhypergraph \(\ho\subseteq\cH\) with
    \(e(\ho) \geq (1-\delta)e(\cH)\) such that \(G[\ho]\notRamsey\),
    then for some \(i \in [t]\) we have \(X_i \leq \delta e(\cH)\)
    and \(S_i\sqsubset \cH\).
  \end{claim}

  \begin{proofclaim}
    Suppose that such a hypergraph \(\ho\) exists.  Then there is a
    red-blue colouring of the edges of $G=G[\ho]$ that contains no red
    copy of $K_{s}$ and no blue copy of $K_{k-1}$.  Since each
    hyperedge \(A \in E(\ho)\) gives us an $s$-clique in \(G\), there
    exists at least one blue edge \(e_A \in E(G)\) that lies in \(A\).
    Let \(G_0\subseteq G\) be the spanning subgraph of $G$ obtained by
    selecting one blue edge inside of each hyperedge of \(\ho\), that
    is, \(E(G_0) = \{e_A : A \in E(\ho)\}\).  Note that
    \(G_0 \sqsubset \ho\).  Now, we must have that \(G_0\) is
    $K_{k-1}$-free, otherwise we would have a blue copy of \(K_{k-1}\)
    in \(G\).

    By Lemma~\ref{containers}, we have
    $S_i \subseteq E(G_0) \subseteq C_i$, for some \(i \in [t]\).
    Furthermore, since all pairs of vertices in \(A \in \cA_i\) do not
    belong to \(C_i\), we cannot have any set in \(\cA_i\) as a
    hyperedge in \(\ho\).  Therefore, the number of sets in \(\cA_i\)
    that are hyperedges in \(\cH\) is at most
    \(|E(\cH)\setminus E(\ho)|\), which implies
    \(X_i \leq e(\cH)-e(\ho) \leq \delta e(\cH)\).  Finally, since
    \(S_i \subseteq E(G_0)\) and \(G_0 \sqsubset \ho \subseteq \cH\),
    we have \(S_i \sqsubset \cH\).
  \end{proofclaim}

  Note that the events $X_i \leq \delta e(\cH)$ and
  $S_i \sqsubset \cH$ are independent, as the first event depends only
  on the sets of \(s\) vertices that are in $\cA_i$ and the second
  event depends only on the sets of $s$ vertices for which a pair of
  vertices is an edge in \(S_i\); since no set \(A \in \cA_i\) can
  have two vertices \(x,y\in A\) with \(xy \in S_i\) (not even with
  $xy \in C_i$), those two events depend on different sets of \(s\)
  vertices.  Therefore, we can bound the probability of the existence
  of $\ho\subseteq \cH$ such that $e(\ho) \geq (1-\delta)e(\cH)$ and
  $G[\ho]\notRamsey$ as follows.
  \begin{align}
    \PP{\exists i \in [t]: X_i \leq \delta e(\cH) \text{ and } S_i \sqsubset \cH}
    & \leq \sum_{i=1}^{t} \PP{X_i \leq \delta e(\cH) \text{ and } S_i \sqsubset \cH} \nonumber \\
    & = \sum_{i=1}^{t} \PP{X_i \leq \delta e(\cH)} \cdot \PP{S_i \sqsubset \cH}.\label{eq:bound-main}
  \end{align}

Note that \(e(\cH) \leq \pp n^{s}/2\) with high
probability.  On the other hand, $X_i$ is a binomial random variable
with the expectation
$\EE{X_i} = \pp |\cA_i| \geq \delta \pp n^{s} \geq 2\delta e(\cH)$.
Therefore, using Chernoff's inequality, we have
\begin{equation}
\PP{X_i \leq \delta e(\cH)}
\leq \PP{X_i \leq \EE{X_i}/2}
\leq \exp\left\{-\delta \pp n^{s}/8 \right\}.\label{eq:bound-xi}
\end{equation}

From Lemma~\ref{lem:quasi}, we have
$\PP{S_i \sqsubset \cH} \leq \pt^{|S_i|}$
for $\pt = \pp \binom{n-2}{s-2}$. Let $m = Dn^{2- 1/m_2(K_{k-1})}$ and
note that from the choice of $C$ we have
$m \leq (D/C)\pp n^{s} \leq \pt n^ 2$. Since $|S_i| \leq m$ for every
$i \in [t]$ and there are at most ${n^2\choose \ell}$ sources $S_i$
with exactly $\ell$ edges, we have
  $$
  \sum_{i=1}^{t} \mathbb{P}[S_i \sqsubset \cH]
  \leq \sum_{i=1}^{t} \pt^{|S_i|}
  \leq \sum_{\ell=1}^{m} \binom{n^2}{\ell} \pt^\ell
  \leq \sum_{\ell=1}^{m} \left(\frac{e\pt n^2}{\ell} \right)^{\ell}.
  $$
  Since $(e\pt n^2/\ell)^{\ell}$ is increasing for
  $\ell \leq \pt n^2$, we may replace $m$ with its upper bound
  $(D/C)\pp n^{s}$ in the above estimation. This together with
  $\pt n^2 \leq \pp n^s$ gives 
\begin{equation}
    \sum_{i=1}^{t} \mathbb{P}[S_i \sqsubset \cH]
    \leq m  \left(\frac{e\pt n^2}{m} \right)^{m}
    \leq n^2 {\left(\frac{eC}{D}\right)}^{(D/C)\pp n^{s}}
    \leq \exp\left(\delta\pp n^{s}/16\right)\label{eq:bound-si},
\end{equation}
where the last inequality follows from the fact that $C$ is
sufficiently large. Finally, using \eqref{eq:bound-xi} and
\eqref{eq:bound-si}, the bound on \eqref{eq:bound-main} becomes
  \begin{align*}
    \PP{\exists i \in [t]: X_i \leq \delta e(\cH) \text{ and } S_i \sqsubset \cH}
    & \leq \exp\left\{-\frac{\delta \pp n^{s}}{16} \right\} = o(1).
  \end{align*}
  Therefore, with high probability, every $\ho\subseteq \cH$ with
  $e(\ho) \geq (1-\delta)e(\cH)$ is such that
  $G[\ho]\rightarrow (K_{s},K_{k-1})$, which finishes the proof.
\end{proof}

\section{$k$-conformal hypergraphs}
\label{sec:hyper}
A hypergraph $\cH$ is \emph{linear} if every pair of
hyperedges of \(\cH\) share at most one vertex. Furthermore, a hypergraph
\(\cH\) is \emph{\(k\)-conformal} if every clique of size exactly $k$ in the
primal graph $G[\cH]$ is contained in a hyperedge of $\cH$. This notion of $k$-conformal hypergraph is inspired by the well-known concept of \emph{conformal hypergraph}, which was introduced by Berge~\cite{Berge}.

Let $\cH$ be a hypergraph and let $S\subseteq V(\cH)$. A family \(\cC =
\{V_1,\ldots,V_\ell\}\) of distinct subsets of \(S\) is a \emph{pair-cover} of
\(S\) if for every \(\{x,y\}\subseteq S\), we have \(\{x,y\}\subseteq V_i\) for
some \(i \in [\ell]\). We say that $\cC = \binom{S}{2}$ is the \emph{perfect}
pair-cover of \(S\). A pair-cover \(\cC\) of \(S\) is \emph{non-trivial} if
\(\cC \neq \{S\}\).
Finally, the \emph{pair-trace} of \(\cH\) on \(S\) is the
family \(\cH_S = \{ E \cap S : E \in E(\cH) \;\text{and}\; |E\cap S|\geq 2\}\).

The next theorem states that with high probability one can obtain a
linear $k$-conformal hypergraph $\ho$ by removing only a small
fraction of the hyperedges of $\hnp$, as long as $\pp$ is much smaller
than a threshold prescribed by the maximum 2-density of $K_{k}$ but
still bigger than $n^{-s}$.

\begin{theorem}\label{k-pair-Helly}
  For all integers $s \geq k \geq 3$, the following holds with high probability for $\cH = \hnp$ when $n^{-s} \ll \pp \ll n^{2-s - 1/m_2(K_k)}$.
  There exists a linear $k$-conformal hypergraph \(\ho \subseteq \cH\) with \((1 - o(1))e(\cH)\) hyperedges.
\label{lemma:cored}
\end{theorem}
\begin{proof}[Proof of Theorem~\ref{lemma:cored}]
  Fix integers $s\geq k \geq 3$, let $n$ be sufficiently large and let
  \(\cH = \hnp\) with \(n^{-s} \ll \pp \ll n^{2-s - 1/m_2(K_k)}\).

Let $S\subseteq V(\cH)$ be a set of exactly $k$ vertices. If $S$ is the vertex
set of a clique in the primal graph $G[\cH]$, then there exists a collection
$\cE = \{E_1,\ldots,E_\ell\}$ of hyperedges of $\cH$, with $1\leq \ell \leq
\binom{|S|}{2}$, such that $\cC = \{E_1\cap S,\ldots,E_\ell \cap S\}$ is a
pair-cover of $S$. Furthermore, if $S$ is not contained in any hyperedge of
$\cH$, then such $\cC$ is a non-trivial pair-cover of $S$. Note that we may
assume that $E_i\cap S \neq E_j \cap S$, for every $i\neq j$, since otherwise we
can remove one of the hyperedges $E_i$ or $E_j$ from $\cE$ and still have a
pair-cover of $S$.

Given a non-trivial pair-cover $\cC = \{V_1,\ldots,V_\ell\}$ of $S$, let
$X_{\cC}$ be the random variable that counts the number of collections $\cE =
\{E_1,\ldots,E_\ell\}$ of hyperedges of $\cH$ such that $E_i\cap S = V_i$, for
every $i\in[\ell]$. Note that for each $i\in[\ell]$, the number of possible
choices for $E_i$ is at most $\binom{n - |S|}{s - |V_i|}$. Therefore, \begin{align*}
  \EE{X_{\cC}} = \prod_{i=1}^{\ell}p \binom{n - |S|}{s - |V_i|}
  \leq p^\ell n^{\sum_{i=1}^{\ell} (s - |V_i|)}
  = (pn^{s-2})^{\ell} n^{-\sum_{i=1}^{\ell} (|V_i|-2)}.
\end{align*}
Using that $p \ll n^{2-s - 1/m_2(K_k)}$, we have
\begin{align}\label{eq:expectation-upperbound}
  \EE{X_{\cC}}
  \ll (pn^{s-2}) \cdot n^{-\frac{\ell-1}{m_2(K_k)} - \sum_{i=1}^{\ell} (|V_i| - 2)}
  = (pn^{s-2}) \cdot n^{-\alpha(\cC)},
\end{align}
where we define
$\alpha(\cC) = \frac{\ell-1}{m_2(K_k)} + \sum_{i=1}^{\ell} (|V_i| -
2)$.

Note that if $\cC$ is the perfect pair-cover of $S$, then
$\alpha(\cC) = k-2$.  The next claim shows that this is in fact a
lower bound for $\alpha(\cC)$ for any non-trivial pair-cover $\cC$ of $S$.

\begin{claim}\label{claim:nice-upper-bound}
  For any non-trivial pair-cover $\cC = \{V_1,\dots,V_\ell\}$ of $S$, we have
  \begin{align*}
    \alpha(\cC) \geq k-2.
  \end{align*}
\end{claim}
\begin{proofclaim}[Proof of Claim~\ref{claim:nice-upper-bound}]
Let $\cC = \{V_1,\dots,V_\ell\}$ be a non-trivial pair-cover of $S$. Without lost of
generality, we may assume that \(|V_1| \geq |V_2| \geq \cdots \geq |V_\ell|\geq
2\). Let $\cC_0 = \cC$ and for each $i\in[\ell]$, inductively define $\cC_i =
(\cC_{i-1} \setminus \{V_i\}) \cup \binom{V_i}{2}$. Note that each $\cC_i$ is a
pair-cover of $S$ and that \(|\cC_i| = |\cC_{i-1}| + \binom{|V_i|}{2} - 1\).
Furthermore, $\cC_{\ell}$ is the perfect pair-cover of $S$. We will show that
$\alpha(\cC_i) \leq \alpha(\cC_{i-1})$ for every $i\in[\ell]$. Indeed, we have
\begin{align*}
  \alpha(\cC_i)
  & = \frac{|\cC_i|-1}{m_2(K_k)} + \sum_{V \in \cC_i} (|V| - 2) \\
  & = \frac{|\cC_{i-1}|-1}{m_2(K_k)} + \sum_{V \in \cC_{i-1}} (|V| - 2) + \frac{\binom{|V_i|}{2} - 1}{m_2(K_k)} - (|V_i| - 2)  \\
  & \leq \alpha(\cC_{i-1}),
\end{align*}
where in the last inequality we used the fact that
$m_2(K_k) \geq \frac{\binom{t}{2}-1}{t-2}$, for any $3 \leq t \leq
k$. Therefore, we conclude that
\begin{align*}
  \alpha(\cC) = \alpha(\cC_0) \geq \alpha(\cC_1) \geq \cdots \geq \alpha(\cC_\ell) = \frac{\binom{|S|}{2}-1}{m_2(K_k)} = k-2,
\end{align*}
finishing the proof.
\end{proofclaim}

Let $X$ be the random variable that counts the number of sets $\cE =
\{E_1,\ldots,E_\ell\}$ of hyperedges of $\cH$ that induce a non-trivial
pair-cover of a set $S$ of size exactly $k$. Note that $X$ can be expressed as a
sum of $X_{\cC}$ over all non-trivial pair-covers $\cC$ of all sets $S$ of size
at most $k$. There are $O(n^k)$ choices for the set $S$ and at most
$2^{k^3}=O(1)$ non-trivial pair-covers of $S$. Therefore,
by~\eqref{eq:expectation-upperbound} and Claim~\ref{claim:nice-upper-bound}, we
have
\begin{align}\label{eq:EX}
    \EE{X}
    \ll n^{k} \cdot  \pp n^{s - 2} \cdot n^{-(k-2)}
    = pn^s.
\end{align}

Now, let $Y$ be the number of pairs of hyperedges in $\cH$ sharing at least two
vertices. Since $\pp \ll n^{2-s}$, we have
  \begin{equation}\label{eq:EY}
    \EE{Y} \le \binom{n}{s}\binom{s}{2}\binom{n-2}{s-2} \pp^2
    \leq s^2 \pp^2 n^{2s-2}\ll pn^s.
  \end{equation}

Finally, since $pn^s\gg 1$, a simple application of Chernoff's inequality gives
that with high probability we have \(e(\cH) = (1 \pm o(1)) \pp \binom{n}{s}\).
From Markov's inequality, we conclude from (\ref{eq:EX}) and (\ref{eq:EY}) that
with high probability $X \ll e(\cH)$ and $Y \ll e(\cH)$. Therefore, by removing
one hyperedge from every set of hyperedges that induces a non-trivial pair-cover
$\cC$ of $S$ counted by $X$ and removing one hyperedge from every pair of
hyperedges counted by $Y$, we obtain a linear $k$-conformal hypergraph $\ho
\subseteq \cH$ that contains $(1-o(1))e(\cH)$ hyperedges.
\end{proof}

As a corollary of Theorem~\ref{k-pair-Helly}, we obtain the following
result.

\begin{theorem}\label{thm:part2}%
Let $k \geq 3$ and $s = R(k)-1$. Then the following holds with high probability
for $\cH = \hnp$ when $n^{-s} \ll \pp \ll n^{2-s - 1/m_2(K_k)}$.
There exists a subhypergraph $\ho \subseteq \cH$ with \(e(\ho) =
(1-o(1))e(\cH)\) such that
    \begin{align}
          G[\ho]\nrightarrow K_k.
    \end{align}
\end{theorem}

\begin{proof}
  Let $\ho \subseteq \cH$ be the linear $k$-conformal hypergraph
  obtained in Theorem~\ref{k-pair-Helly} and let $G=G[\ho]$. To verify that
  $G \nrightarrow K_k$, we colour the edges of $G$ as
  follows: for each hyperedge $E \in E(\ho)$, since $|E|=s=R(k)-1$, we can colour all the edges of $G$ contained in $E$ in a way that there is no monochromatic copy of $K_k$. Considering that $\ho$ is linear,
  every edge of $G$ belongs to exactly one hyperedge of $\ho$ and
  hence this colouring is well-defined.
  Now, since $\ho$ is $k$-conformal, every set of $k$ vertices that induces
  a copy of $K_k$ must be contained in some hyperedge of $\ho$ and it cannot be monochromatic.
  Therefore, $G \nrightarrow K_k$, as desired.
\end{proof}

\section{Proof of Theorem~\ref{thm:main}}
\label{sec:proof}

In this short section we combine Theorems~\ref{thm:part1}
and~\ref{thm:part2} to prove our main result.
\begin{proof}[Proof of Theorem~\ref{thm:main}]
  Let $k \geq 3$ be an integer and $s = R(k)-1$. Consider $\pp$ such that $n^{2-s-1/m_2(K_{k-1})} \ll \pp \ll n^{2-s - 1/m_2(K_k)}$ and let $\cH = \hnp$.
  By Theorem~\ref{thm:part1}, with high probability, every subhypergraph $\ho \subseteq \cH$ with \(e(\ho) = (1-o(1))e(\cH)\) satisfies
  \begin{align}
    G[\ho]\rightarrow(K_{s},K_{k-1}).
    \label{eq:H0a}
  \end{align}
  On the other hand, by Theorem~\ref{thm:part2}, with high probability
  there exists a subhypergraph $\ho \subseteq \cH$ with
  \(e(\ho) = (1-o(1))e(\cH)\) such that
  \begin{align}
    G[\ho]\nrightarrow K_k.
    \label{eq:H0b}
  \end{align}
  Since both events can occur with high probability, there exists a
  hypergraph $\ho$ such that both~\eqref{eq:H0a}
  and~\eqref{eq:H0b} hold. Therefore, $G[\ho]$ is the desired graph.

\end{proof}

\section{Concluding Remarks}
\label{sec:conc}

In this work, we constructed, for every integer $k \geq 3$, a graph
$G$ such that $G \nrightarrow K_k$, but
$G \rightarrow (K_{s},K_{k-1})$, for $s=R(k)-1$. Our approach combines
probabilistic techniques with hypergraph containers to obtain
``pseudorandom'' host graphs that exhibit some particular Ramsey
behavior. This way one can encode the construction of $G$ through a
random $s$-uniform hypergraph $\cH$, obtained by creating copies of
$K_s$ for each hyperedge of $\cH$, which is carefully pruned to
eliminate certain configurations that could otherwise lead to
monochromatic copies of $K_k$. The resulting graph $G$ simultaneously
avoids monochromatic copies of $K_k$ in some colouring while forcing
either a red copy of $K_s$ or a blue copy of $K_{k-1}$ in any red-blue
colouring of the edges of $G$.

There are several directions for future work. It would be interesting
to find deterministic constructions of such graphs, or to impose
additional structural constraints such as ``bounded'' degree or
forbidding some subgraphs. More broadly, a natural question is to
determine for which values of $k$ the inequality $R(k{-}1,k{+}1) <
R(k)$ is strict, and whether methods similar to ours can help
characterizing more generally when asymmetric pairs are not
Ramsey-equivalent to the corresponding diagonal pair.

It is possible to adapt our proof to obtain the following
generalization of Theorem~\ref{thm:main} by considering a linear
$k$-conformal subhypergraph of $\cH_{s}(n,p)$, by choosing
$n^{2-s-1/m_2(K_{\ell-1})} \ll p \ll n^{2-s-1/m_2(K_{\ell})}$.
\begin{theorem}\label{thm:generalization}
  For any integers $k \ge \ell \ge 3$, there exists a graph $G$ such
  that $G \nrightarrow (K_k,K_{\ell})$ and
  $G \rightarrow (K_{s},K_{\ell-1})$ for $s\leq R(k,\ell)-1$.
\end{theorem}
We propose the following conjecture as a variation of the previous
theorem for three colours.
\begin{conjecture}\label{conj:first}
  For any integers $k \ge \ell \ge 2$, there exists a graph $G$ such
  that $G \nrightarrow (K_k,K_{\ell})$ and
  $G \rightarrow (K_{k-1},K_{k-1},K_{\ell})$.
\end{conjecture}
Note that the case $\ell = 2$ of the above conjecture is precisely the
result of Ne\v{s}et\v{r}il and R\"odl (Theorem~\ref{thm_NR}). We
conclude proposing the following conjecture that relates to
Conjecture~\ref{conj:first} in the same way that
Theorem~\ref{thm:generalization} relates to Theorem~\ref{thm_NR}.
\begin{conjecture}
  For any integers $k \ge \ell \ge 2$, there exists a graph $G$ such
  that $G \nrightarrow (K_k,K_k,K_{\ell})$, but
  $G \rightarrow (K_{k+1},K_{k-1},K_{\ell})$.
\end{conjecture}

\subsection*{Acknowledgements.}
We would like to thank Yoshiharu Kohayakawa for helpful discussions
and for his comments on an earlier version of this manuscript.

\printbibliography
\end{document}